\documentclass{aptpub2}

\usepackage{bbm}
\usepackage{xcolor}

\authornames{NATALIA SOJA-KUKIE\L A}
\shorttitle{On maxima of stationary fields}

\renewcommand{\E }{\mathbb{E}}
\renewcommand{\P}{\mathbb{P}}

\newcommand{\w}[1]{\mathbf{#1}}

\newcommand{\GN}{{\mathbb N}}
\newcommand{\GR}{{\mathbb R}}
\newcommand{\GZ}{{\mathbb Z}}
\newcommand{\bn}{{\mathbf n}}
\newcommand{\bN}{{\mathbf N}}
\newcommand{\bp}{{\mathbf p}}

\newcommand{\bk}{{\mathbf k}}
\newcommand{\bj}{{\mathbf j}}
\renewcommand{\bi}{{\mathbf i}}
\newcommand{\bs}{{\mathbf s}}

\theoremstyle{plain}
\newtheorem{theo}{Theorem}[section]

\newtheorem{corol}[theo]{Corollary}

\newtheorem{rema}[theo]{Remark}

\begin{document}

\title{On maxima of stationary fields}

\authorone[Nicolaus Copernicus University in Toru\'n, Poland]{N. Soja-Kukie\L a}

\addressone{Faculty of Mathematics and Computer Science, Nicolaus Copernicus University, ul.~Chopina 12/18, \mbox{87-100 Toru\'n}, Poland; e-mail address: \texttt{natas@mat.umk.pl}}

\begin{abstract}
Let $\{X_{\bn} : \bn\in\GZ^d\}$ be a weakly dependent stationary random field with maxima $M_{A} := \sup\{X_{\bi} : \bi\in A\}$ for finite $A\subset\GZ^d$ and \mbox{$M_{\bn} := \sup\{X_{\bi} : \w{1} \leq  \bi \leq  \bn \}$} for $\bn\in\GN^d$. 
In a~general setting we prove that \mbox{$\P(M_{(N_1(n),N_2(n),\ldots, N_d(n))} \leq  v_n)$} $= \exp(- n^d \P(X_{\w{0}} > v_n , M_{A_n} \leq  v_n)) +  o(1)$ for some increasing sequence of sets $A_n$ of size $ o(n^d)$, where $(N_1(n),N_2(n), \ldots,N_d(n))\to(\infty,\infty, \ldots, \infty)$ and $N_1(n)N_2(n)\cdots N_d(n)\sim n^d$. The sets $A_n$ are determined by a translation invariant total order $\preccurlyeq$ on $\GZ^d$.
For a class of fields satisfying a~local mixing condition, including $m$-dependent ones, the main theorem holds with a constant finite $A$ replacing $A_n$. The above results lead to new formulas for the extremal index for random fields. The new method of calculating limiting probabilities for maxima is compared with some known results and applied to the moving maximum field.
\end{abstract}



\section{Introduction}

Let us consider a $d$-dimensional stationary random field $\{X_\bn : \bn\in\GZ^d\}$ with its partial maxima
\[ M_A := \sup \{ X_\bi  : \bi \in A\}\]
defined for finite $A\subset\GZ^d$. We also put
$M_{\bj,\bn} := \sup \{ X_\bi  : \bj \leq  \bi  \leq  \bn\}$
and $M_{\bn} := M_{\w{1},\bn}$ for $\bj,\bn\in\GZ^d$.
The goal is to study the asymptotic {behaviour} of $\P(M_{\bN(n)} \leq  v_n)$ as $n\to\infty$, for $\{v_n\}\subset \mathbb{R}$ and $\bN(n)\to\pmb{\infty}$ coordinatewise.

In the case $d=1$, when $\{X_n: n\in\mathbb{Z}\}$ is a stationary sequence, the well known result of O'Brien \cite[Theorem 2.1]{OBr87} states that under a broad class of circumstances
\begin{equation}\label{MAX_1}
\P(M_{n}\leq  v_n) = \exp(-n \P(X_0>v_n, M_{p(n)}\leq  v_n)) +  o(1)
\end{equation}
holds for some $p(n)\to\infty$ satisfying $p(n)= o(n)$. For \mbox{$m$-dependent} $\{X_n\}$ we can set $p(n):=m$ in formula (\ref{MAX_1}), as Newell \cite{Ne64} shows. 
It follows that the extremal index $\theta$ for $\{X_n\}$, defined by Leadbetter \cite{Lea83}, equals
\begin{equation}\label{THETA_1}
\theta = \lim_{n\to\infty} \P(M_{p(n)}\leq  v_n \,|\, X_0>v_n),
\end{equation}
where $p(n)=m$ in the $m$-dependent case. More generally, we can put $p(n)=m$ in (\ref{THETA_1}) whenever Condition $D^{(m+1)}(v_n)$, introduced by Chernick et al. \cite{CHM91}, is satisfied.

We recall that the extremal index $\theta\in[0,1]$ is interpreted as the reciprocal of the mean number of high threshold exceedances in a cluster. Formula (\ref{THETA_1}) for $\theta$ may be treated as an answer to the question: {\it Asymptotically, what is the probability that a~given element of a~cluster of large values is its last element on the right?}

Looking for formulas analogous to (\ref{MAX_1}) and (\ref{THETA_1}) for arbitrary $d\in\mathbb{N_+}$, one can try to answer the properly formulated $d$-dimensional version of the above question. This point is realized in Sections \ref{MAIN} and \ref{SEC_EI}. In Section \ref{MAIN} we prove the main result, Theorem~\ref{THEO_MAIN}. We establish that in a general setting the approximation
\begin{equation}\label{OUR_FORMULA}
\P(M_{\bN(n)} \leq  v_n) = \exp\left(-n^d \P\left(X_{\w{0}} > v_n, M_{A(\bp(n))} \leq  v_n\right)\right) +  o(1)
\end{equation}
holds with $A(\bp(n)) \subset \{\bj \in \GZ^d : -\bp(n) \leq \bj \leq \bp(n)\}$ defined by (\ref{DEF_A}), $\bN(n)$ fulfilling (\ref{N_PROP}) and $\bp(n)\to\pmb{\infty}$ satisfying $\bp(n)= o(\bN(n))$ and some other rate of growth conditions. For $d=1$ we have $A(p(n))=\{1,2,\ldots, p(n)\}$ and formula (\ref{OUR_FORMULA}) simplifies to~(\ref{MAX_1}). Corollary~\ref{COR_LOCAL} provides the local mixing condition (\ref{COND_DM}) equivalent to the fact that (\ref{OUR_FORMULA}) holds with $\bp(n):=(m,m,\ldots,m)$.  Section \ref{SEC_EI} is devoted to considerations concerning the notion of the extremal index for random fields. Formula (\ref{F1}), being a~generalization of (\ref{THETA_1}), and its simplified version (\ref{F2}) for fields fulfilling (\ref{COND_DM}) are proposed there. In Section \ref{MM_SEC} the results from Sections \ref{MAIN} and \ref{SEC_EI} are applied to describe the asymptotics of partial maxima for the moving maximum field.

In Section \ref{SEC_M} we focus on $m$-dependent fields and present a corollary of the main theorem generalizing the aforementioned Newell's \cite{Ne64} formula. We also compare the obtained result with the limit theorem for $m$-dependent fields proven by Jakubowski and Soja-Kukie\l a \cite[Theorem~2.1]{JS}.

The present paper provides a $d$-dimensional generalization of O'Brien's formula (\ref{MAX_1}) with a handy and immediate conclusion for \mbox{$m$-dependent} fields. Another general theorem by Turkman \cite[Theorem 1]{Tu06} is not well applicable in the $m$-dependent case.
A recent result obtained independently by Ling \cite[Lemma 3.2]{LING} is a special case of Theorem \ref{THEO_MAIN}.
Other theorems on the topic were given for some subclasses of weakly dependent fields: in the 2-dimensional Gaussian setting by French and Davis~\cite{FD}; for 2-dimensional moving maxima {and} moving averages by Basrak and Tafro \cite{B-T2014}; for \mbox{$m$-dependent} and max-$m$-approximable fields by Jakubowski and Soja-Kukie\l a \cite{JS}; for regularly varying fields by Wu and Samorodnitsky \cite{WS}.  The proof of Theorem~\ref{THEO_MAIN} presented in the paper, although achieved independently, is similar to proofs of \cite[Lemma 4]{FD} and \cite[Lemma 3.2]{LING}.

\section{Preliminaries}\label{PRE}

An~element $\bn\in\GZ^d$ is often denoted by $(n_1,n_2,\ldots,n_d)$ and $\|\bn\|$ is its sup norm. We write $\bi\leq  \bj$ and $\bn\to\pmb{\infty}$, whenever $i_l\leq  j_l$ and $n_l\to\infty$, respectively, for all $l\in\{1,2,\ldots, d\}$. We put $\mathbf{0}:=(0,0,\ldots,0)$, $\mathbf{1}:=(1,1,\ldots,1)$ and $\pmb{\infty}:=(\infty,\infty,\ldots,\infty)$.

In our considerations $\{X_{\bn} : \bn\in\GZ^d\}$ is a $d$-dimensional stationary random field. We ask for the asymptotics of $\P(M_{\bN(n)}\leq  v_n)$ as $n\to\infty$, for
$\bN = \{\bN(n) : n\in\GN\}\subset\GN^d$
such that
\begin{equation}\label{N_PROP}
\bN(n)\to\pmb{\infty}\;\;\text{and} \;\; N^*(n):=N_1(n)N_2(n)\cdots N_d(n)\sim n^d
\end{equation}
and $\{v_n : n\in\GN\} \subset \GR$. 

We are interested in {\em weakly dependent} fields. In the paper we assume that
\begin{eqnarray}\label{WEAK_DEP}
\P\left(M_{\bN(n)}\leq  v_n\right) = \P\left(M_{\bp(n)}\leq  v_n\right)^{k_n^d}+ o(1)
\end{eqnarray}
is satisfied for some $r_n\to\infty$ and all $k_n\to\infty$ such that $k_n= o(r_n)$,
with
\begin{equation}\label{DEF_p}
\bp(n):= \left(\lfloor N_1(n)/k_n\rfloor,\lfloor N_2(n)/k_n\rfloor,\ldots,\lfloor N_d(n)/k_n\rfloor \right).
\end{equation}
Applying the classical fact (see, e.g., O'Brien \cite{OBr87}):
\begin{equation}\label{FACT_OB}
(a_n)^n-\exp(-n(1-a_n))\to 0 \quad\text{as}\quad n\to\infty,\quad \text{for} \quad a_n\in[0,1],
\end{equation}
we get that (\ref{WEAK_DEP}) implies
\begin{eqnarray}\label{FOR}
\P\left(M_{\bN(n)}\leq  v_n\right) = \exp\left(-k_n^d \P\left(M_{\bp(n)}> v_n\right)\right)+ o(1).
\end{eqnarray}
Above, $p_l(n)= o(N_l(n))$ for $l\in\{1,2,\ldots,d\}$, which we briefly denote $\bp(n)= o(\bN(n))$.

\begin{rema}
For $d=1$ weak dependence in the sense of (\ref{WEAK_DEP}) is ensured by any of the following conditions: Leadbetter's $D(v_n)$, O'Brien's $AIM(v_n)$ or Jakubowski's $B_1(v_n)$; see~\cite{Lea83,OBr87,J91}.  For $d\in\GN_+$ the considered property follows for example from Condition $B_1^{\bN}(v_n)$ introduced by Jakubowski and Soja-Kukie\l a \cite{JS_PDF}. In particular, $m$-dependent fields are weakly dependent; see Section \ref{SEC_M}. A similar notion of weak dependence was investigated by Ling \cite[Lemma 3.1]{LING}.
\end{rema}

Let $\preccurlyeq$ be an arbitrary total order on $\GZ^d$ which is {\it translation invariant}, i.e. $\bi \preccurlyeq \bj$ implies $\bi + \bk \preccurlyeq \bj + \bk$ . An~example of such an order is the lexicographic order:
$$ \bi \preccurlyeq \bj \qquad \text{iff} \qquad (\bi = \bj \quad \text{or} \quad i_l < j_l \text{ for the first } l \text{ where } i_l \text{ and } j_l  \text{ differ}).$$
We will write $\bi \prec \bj$ whenever $\bi \preccurlyeq \bj$ and $\bi\neq\bj$.
For technical needs of further sections, we define the set $A(\bp)\subset \GZ^d$ for each $\bp\in\GN^d$ as follows:
\begin{equation}\label{DEF_A}
A(\bp) := \left\{\bj\in\GZ^d : -\bp \leq \bj \leq \bp \text{ and } \w{0} \prec \bj \right\}.
\end{equation}
\section{Main theorem}\label{MAIN}
In the following the main result of the paper is presented. The asymptotic behaviour of $\P(M_{\bN(n)}\leq  v_n)$ as $n\to\infty$, for weakly dependent $\{X_{\bn}\}$ and for $\{\bN(n)\}$ and $\{v_n\}$ as in Section~\ref{PRE}, is described.
\begin{theo}\label{THEO_MAIN}
Let $\{X_\bn\}$ satisfy (\ref{WEAK_DEP}) for some $r_n\to\infty$ and all $k_n\to\infty$ such that $k_n=o(r_n)$. If
\begin{equation}\label{ASS}
\liminf_{n\to\infty} \P\left(M_{\bN(n)} \leq  v_n\right) > 0,
\end{equation}
then for every $\{k_n\}$ as above
\begin{equation}\label{CONV_MAIN}
\P\left(M_{\bN(n)} \leq  v_n\right) = \exp\left(- n^d \P\left(X_{\w{0}} > v_n, M_{A(\bp(n))}\leq  v_n\right)\right) +  o(1)
\end{equation}
holds with $\bp(n)$ and $A(\bp(n))$ given by (\ref{DEF_p}) and (\ref{DEF_A}), respectively.
\end{theo}

\begin{rema}
If (\ref{WEAK_DEP}) holds for some $k_n\to\infty$, then (\ref{ASS}) is implied by the condition
\begin{equation}\label{ASS_STRONGER}
\limsup_{n\to\infty} n^d\P(X_{\w{0}} > v_n) < \infty.
\end{equation}
This follows from (\ref{FOR}) and the inequality 
$$k_n^d\P\left(M_{\bp(n)}> v_n\right) \leq N^*(n) \P(X_{\mathbf{0}} >v_n) \sim n^d \P(X_{\mathbf{0}} >v_n).$$
\end{rema}

The proof of the theorem generalizes the reasoning proposed by O'Brien \cite[Theorem 2.1]{OBr87} for sequences. A~way of dividing the event $\{M_{\bp(n)} > v_n\}$ into $p^*(n):=p_1(n)p_2(n)\cdots p_d(n)$ mutually exclusive events determined by $\preccurlyeq$ (which are similar in some sense) plays a key role in the proof. An analogous technique was used by French and Davis \cite[Lemma~4]{FD} in the 2-dimensional Gaussian case. Recently, Ling \cite[Lemma~3.1]{LING} expanded their result to some non-Gaussian fields. In both papers the authors restrict themselves to the lexicographic order on $\GZ^2$.

\begin{proof}[Proof of Theorem \ref{THEO_MAIN}]
Let the assumptions of the theorem be satisfied. Then (\ref{FOR}) holds.
Dividing the set $\{M_{\bp(n)} > v_n\}$ into $p^*(n)=p_1(n)p_2(n)\cdots p_d(n)$ disjoint sets and applying monotonicity and stationarity we obtain that
\begin{eqnarray*}
\lefteqn{\P(M_{\bp(n)} > v_n)}\\
& = & \sum_{\mathbf{1}\leq \bj \leq \bp(n)} \P\left(X_\bj > v_n, X_{\bi} \leq  v_n \text{ for all } \bi \succ \bj \text{ such that } \w{1} \leq \bi \leq  \bp(n) \right)\\
& \geq  & \sum_{\mathbf{1}\leq \bj \leq \bp(n)} \P\left(X_\bj > v_n, X_{\bi} \leq  v_n \text{ for all } \bi \in A(\bp(n))+\bj\right) \\
& = & p^*(n) \P\left(X_{\w{0}} > v_n, M_{A(\bp(n))} \leq  v_n\right),
\end{eqnarray*}
which combined with (\ref{FOR}) and the fact that $k_n^dp^*(n) \sim n^d$ gives
\begin{equation}\label{IN_1}
\P\left(M_{\bN(n)}\leq  v_n\right)  \leq   \exp\left(-n^d \P\left(X_{\w{0}} > v_n, M_{A(\bp(n))} \leq  v_n\right)\right) +  o(1).
\end{equation}

In the second step of the proof we shall show that the inequality reverse to (\ref{IN_1}) also holds. It is sufficient to consider the case $\P(M_{\bN(n)} \leq  v_n) \to \gamma$ for $\gamma \in [0,1]$, and we do so. Since $\gamma=0$ is excluded by assumption (\ref{ASS}) and for $\gamma =1$ the proven inequality is obvious, we focus on $\gamma\in(0,1)$. Let us choose $\{t_n\}\subset\GN_+$ so that $t_n\to\infty$ and $t_n=o(k_n)$. Put 
{$\bs(n):= (\lfloor N_1(n)/t_n\rfloor,\lfloor N_2(n)/t_n\rfloor , \ldots, \lfloor N_d(n)/t_n\rfloor )$} 
and $s^*(n):=s_1(n)s_2(n)\cdots s_d(n)$. Since $t_n= o(r_n)$, (\ref{FOR}) holds with $k_n$ replaced by $t_n$ and $\bp(n)$ replaced by $\bs(n)$. Also $\bp(n) =  o(\bs(n))$ and $\bs(n)= o(\bN(n))$. Moreover, for the sets
$$C(\bp(n), \bs(n)):= \left\{\bj \in \GZ^d : \bp(n)+\w{1} \leq  \bj \leq  \bs(n)-\bp(n)\right\}$$
and
$$B(\bp(n), \bs(n)):= \left\{\bj\in\GZ^d : \w{1} \leq  \bj \leq  \bs(n)\right\} \backslash \, C(\bp(n), \bs(n))$$
we obtain that
\begin{eqnarray*}
\frac{\P\left(M_{\bs(n)} > v_n, M_{B(\bp(n), \bs(n))} \leq  v_n\right)}{\P\left(M_{B(\bp(n), \bs(n))} > v_n\right)}
& =  &\frac{\P\left(M_{\bs(n)} > v_n\right) - \P\left(M_{B(\bp(n), \bs(n))} > v_n\right)}{\P\left(M_{B(\bp(n), \bs(n))} > v_n\right)} \\
& =  &\frac{\P\left(M_{\bs(n)} > v_n\right)}{\P\left(M_{B(\bp(n), \bs(n))} > v_n\right)} - 1\\
& = & \frac{\P\left(M_{\bs(n)} > v_n\right)}{ o(s^*(n)/p^*(n)) \cdot \P\left(M_{\bp(n)} > v_n\right)} - 1\\
& = & \frac{1+o(1)}{o(1)} \cdot \frac{t_n^d \P\left(M_{\bs(n)} > v_n\right)}{k_n^d \P\left(M_{\bp(n)} > v_n\right)} - 1.
\end{eqnarray*}
Applying (\ref{FOR}) twice, we get
$$\frac{t_n^d \P\left(M_{\bs(n)} > v_n\right)}{k_n^d \P\left(M_{\bp(n)} > v_n\right)} \to \frac{-\log \gamma}{- \log \gamma} = 1, \quad \text{as} \quad n\to\infty,$$
and consequently
\begin{equation}\label{NEG}
\frac{\P\left(M_{B(\bp(n), \bs(n))} > v_n\right)}{\P\left(M_{\bs(n)} > v_n, M_{B(\bp(n), \bs(n))} \leq  v_n\right)} \to 0, \quad \text{as} \quad n\to\infty.
\end{equation}
Now, observe that
\begin{eqnarray*}
\P\left(M_{\bs(n)} > v_n\right) & = & \P\left(M_{\bs(n)} > v_n, M_{B(\bp(n), \bs(n))} \leq  v_n\right) + \P\left(M_{B(\bp(n), \bs(n))} > v_n\right) \\
& = & \P\left(M_{\bs(n)} > v_n, M_{B(\bp(n), \bs(n))} \leq  v_n\right) (1+ o(1))\\
& \leq  &  \sum_{\bj \in C(\bp(n), \bs(n))} \P\left(X_{\w{j}} > v_n, M_{A(\bp(n)) + \bj} \leq  v_n\right) \cdot (1+ o(1))\\
& \leq & s^*(n) \P\left(X_{\w{0}} > v_n, M_{A(\bp(n))} \leq  v_n\right) (1+ o(1)),
\end{eqnarray*}
by property (\ref{NEG}), subadditivity and monotonicity of probability, and by stationarity of the field $\{X_{\bn}\}$. Applying (\ref{FOR}) with $(k_n, \bp(n))$ replaced by $(t_n, \bs(n))$ and the fact that $t_n^d s^*(n) \sim n^d$, we conclude that
\begin{eqnarray}\label{IN_2}
\P\left(M_{\bN(n)}\leq  v_n\right) & \geq  & \exp\left(-t_n^d s^*(n) \P\left(X_{\w{0}} > v_n, M_{A(\bp(n))} \leq  v_n\right)(1+ o(1))\right) + o(1)\nonumber \\
&=& \exp\left(-n^d \P(X_{\w{0}} > v_n, M_{A(\bp(n))} \leq  v_n)\right) +  o(1).
\end{eqnarray}
Since inequalities (\ref{IN_1}) and (\ref{IN_2}) are both satisfied, the proof is complete.
\end{proof}

Theorem \ref{THEO_MAIN} immediately yields the following generalization of the result established by Chernick et al. \cite[Proposition~1.1]{CHM91} for $d=1$. Assumption (\ref{COND_DM}) given below is a~multidimensional counterpart of the local mixing {\em Condition} $D^{(m+1)}(v_n)$ defined in \cite{CHM91} for sequences and it is satisfied by, e.g., $m$-dependent fields (see Section \ref{M_FIRST}).

\begin{corol}\label{COR_LOCAL}
Let the assumptions of Theorem \ref{THEO_MAIN} be satisfied. Then 
\begin{equation*}
\P\left(M_{\bN(n)} \leq  v_n\right) = \exp\left(- n^d \P\left(X_{\w{0}} > v_n, M_{A((m,m,\ldots,m))}\leq  v_n \right)\right) +  o(1)
\end{equation*}
if and only if
\begin{equation}\label{COND_DM}
n^d \P\left(X_{\w{0}} > v_{n} \geq  M_{A((m,m,\ldots,m))}, M_{A(\bp(n)) \backslash A((m,m,\ldots,m))} > v_n \right) \xrightarrow[n\to\infty]{} 0,
\end{equation}
where $k_n\to\infty$ is such that k$_n= o(r_n)$.
\end{corol}
We point out that Corollary \ref{COR_LOCAL} reforms a faulty formula for $m$-dependent fields proposed by Ferreira and Pereira \cite[Proposition 2.1]{F-P08}; see \cite[Example 5.5]{JS}. We also suggest to compare the above condition (\ref{COND_DM}) with assumption $D''(v_n, \mathcal{B}_n,\mathcal{V})$ proposed by Pereira et al. \cite[Definition 3.1]{PMF}.

\begin{rema}\label{PPA}
There exists a close relationship between Theorem \ref{THEO_MAIN} and compound Poisson approximations in the spirit of Arratia et al. \cite[Section 4.2.1]{AGG90}. The random variable 
$${\Lambda^{(1)}_n}:=\sum_{\mathbf{1}\leq\bk \leq \bN(n)} \mathbbm{1}_{\{X_{\w{k}} > v_n, M_{\bk + A(\bp(n))}\leq  v_n \}},$$
with the expectation ${\lambda^{(1)}_n}:=N^*(n) \P\left(X_{\w{0}} > v_n, M_{A(\bp(n))}\leq  v_n \right)$,
estimates the number of clusters of exceedances over $v_n$ in the set $\{\bk : \mathbf{1} \leq \bk \leq \bN(n)\}$ and we have \mbox{$\P(M_{\bN(n)} \leq v_n) = \exp(-{\lambda^{(1)}_n}) + o(1)$}. 
\end{rema}

\begin{rema}
It is worth noting that translation invariant linear orders on the set of indices~$\GZ^d$ play a significant role in considerations (by Basrak and Planini\'c \cite{BP}, Wu and Samorodnitsky \cite{WS}) on the extremes of regularly varying fields.
\end{rema}

\section{Extremal index}\label{SEC_EI}

In this part we apply the results given in Section \ref{MAIN} to establish formulas (\ref{F1}) and (\ref{F2}) for the extremal index $\theta$ for random fields. We refer to Choi \cite{CHOI} or Jakubowski and Soja-Kukie\l a \cite{JS} for definitions and some considerations on the extremal index in the multidimensional setting.

Here we present a method of calculating the number $\theta\in[0,1]$ satisfying
\begin{equation}\label{THETA_DEF}
\P(M_{\bN(n)} \leq  v_n) - \P(X_{\w{0}} \leq  v_n) ^{\theta n^d} \to 0 \quad \text{as}\quad n\to\infty,
\end{equation}
whenever such $\theta$ exists.
Let us observe that according to (\ref{FACT_OB}) we have
$$\P(X_{\w{0}} \leq  v_n) ^{n^d} = \exp\left(-n^d \P(X_{\w{0}} > v_n)\right) +  o(1)$$
and, moreover, Theorem \ref{THEO_MAIN} yields
$$\P(M_{\bN(n)} \leq  v_n) = \exp\left(- n^d \P\left(X_{\w{0}} > v_n, M_{A(\bp(n))}\leq  v_n\right)\right) + o(1)$$
for $\bN(n)$, $v_n$ and $\bp(n)$ satisfying appropriate assumptions.
Hence, provided that 
$$ 0 < \liminf_{n\to\infty} n^d\P(X_{\w{0}} > v_n) \leq  \limsup_{n\to\infty }n^d\P(X_{\w{0}} > v_n) < \infty,$$
condition (\ref{THETA_DEF}) is satisfied
if and only if
\begin{equation}\label{F1}
\theta = \lim_{n\to\infty}\P\left(M_{A(\bp(n))}\leq  v_n \, \big| \, X_{\w{0}} > v_n\right).
\end{equation}
Formula (\ref{F1}), allowing computation of extremal indices $\theta$ for random fields, is a~multidimensional generalization of (\ref{THETA_1}). In the special case when assumption (\ref{COND_DM}) is satisfied, it is easy to show that formula (\ref{F1}) simplifies to the following one:
\begin{equation}\label{F2}
\theta = \lim_{n\to\infty} \P\left(M_{A((m,m,\ldots,m))}\leq  v_n \, \big| \, X_{\w{0}} > v_n\right).
\end{equation} 

The above formulas are in line with the interpretation of $\theta$ as the reciprocal of the mean number of high threshold exceedances in a cluster.
Indeed, they answer the question: {\it What is the asymptotic probability that a given element of a cluster is the~distinguished element of the cluster?}, where the distinguished element in a cluster is the greatest one with respect to the order $\preccurlyeq$. Such identification of a~unique representative for each cluster is called {\em declustering}, {\em declumping} or {\em anchoring} and has much in common with compound Poisson approximations (see, e.g., \cite{AGG90, BC,BP}).

\begin{rema}
Formula (\ref{F1}) justifies the following definition of the runs estimator $\hat{\theta}^R_{\bN(n)}$ for the extremal index $\theta$:
$$\hat{\theta}^R_{\bN(n)} := S_n^{-1}\sum_{\w{1}+\bp (n) \, \leq \, \bk \, \leq \, \bN(n)-\bp(n)} { \mathbbm{1}_{\{X_{\w{k}} > v_n, M_{\bk + A(\bp(n))}\leq  v_n\}} },$$
where $S_n$ is the number of exceendances over $v_n$ in the set $\{\bk\in\GZ^d : \w{1} \leq \bk \leq \bN(n)\}$.
\end{rema}

\section{Maxima of $m$-dependent fields}\label{SEC_M}

In this section we focus on $m$-dependent fields. We recall that $\{X_\bn\}$ is {\em $m$-dependent} for some $m\in\GN $, if the families $\{X_{\bi}:\bi\in U \}$ and $\{X_{\bj} : \bj\in V\}$ are independent for all pairs of finite sets  $ U , V \subset\GZ^d$ satisfying $\min \{\|\bi-\bj\| : \bi\in  U ,\,\bj\in V \} > m$.

Let us assume that $\{X_{\bn}\}$ is $m$-dependent and satisfies (\ref{ASS}) for some sequence \mbox{$\{v_n\}\subset\GR $}. Then it is easy to show that condition (\ref{ASS_STRONGER}) holds too (see \cite[Remark 4.2]{JS}). Below, we present two methods that can be applied to calculate the limit of $\P(M_{\bN(n)} \leq  v_n)$. A direct connection between them can be given and we illustrate it in the case $d=2$.

\subsection{First method}\label{M_FIRST}
The first of the methods is a consequence of the main result presented in the paper. Since
the field $\{X_{\bn}\}$ is $m$-dependent, it satisfies (\ref{WEAK_DEP}) for each $k_n\to\infty$ such that $k_n= o(r_n)$, for some $r_n\to\infty$
 (see, e.g., \cite{JS}). Moreover, the inequality
\begin{eqnarray*}
\lefteqn{n^d \P\left(X_{\w{0}} > v_{n} \geq  M_{A((m,m,\ldots,m))}, M_{A(\bp(n)) \backslash A((m,m,\ldots,m))} > v_n \right)}\\
& \leq &   n^d \sum_{\bi \in A(\bp(n)), \|\bi\| > m} \P(X_{\w{0}} > v_{n}, X_{\bi} > v_{n})\\
& = & n^d \sum_{\bi \in A(\bp(n)), \|\bi\| > m} \P(X_{\w{0}} > v_{n})\P(X_{\bi} > v_{n})
 \leq  n^d \cdot \frac{n^d}{k_n^d} \cdot \P(X_{\w{0}} > v_{n})^2 (1+ o(1))
\end{eqnarray*}
holds with the right-hand side tending to zero by (\ref{ASS_STRONGER}). From Corollary \ref{COR_LOCAL} we obtain that
\begin{equation}\label{FIRST}
\P\left(M_{\bN(n)} \leq  v_n\right) = \exp\left( - n^d \P(X_{\w{0}} > v_n, M_{A((m,m,\ldots,m))}\leq  v_n\right) +  o(1).
\end{equation}

\subsection{Second method}\label{M_SECOND}
The second formula comes from Jakubowski and Soja-Kukie\l{}a \cite[Theorem 2.1]{JS}. It states that we have
\begin{eqnarray}\label{SECOND}
\P \! \left(M_{\bN (n)} \! \leq \!  v_n\right)=\exp\left( -n^d \!\!\!\! \sum_{\pmb{\varepsilon}\in\{0,1\}^d} \!\! (-1)^{\varepsilon_1+\varepsilon_2+\ldots+\varepsilon_d}\P \! \left( M_{\pmb{\varepsilon},(m,m,\ldots,m)} \! > \! v_n\right) \right) \! + \! o(1)
\end{eqnarray}
under the above assumptions on $\{X_{\bn}\}$.
This result is a consequence of the~Bonferroni-type inequality from Jakubowski and Rosi\'nski \cite[Theorem 2.1]{J-R99}. 

\subsection{Comparison}\label{SEC_M.COMP}
For $d=1$ both of the formulas simplify to the well-known result of Newell \cite{Ne64}:
$$\P(M_n \leq  v_n) = \exp(-n\P(X_0> v_n, M_{1,m} \leq  v_n)) +  o(1).$$
Each of them allows us to describe the asymptotic behaviour of maxima on the base of tail properties of joint distribution of a fixed finite dimension. To apply the first method, one uses the distribution of the $(1 + ((2m+1)^d-1)/2)$-element family \mbox{$\{X_{\bn}: \bn\in \{\w{0}\} \cup A((m,m,\ldots,m))\}$}. To involve the second method, one bases on the distribution of the $(m+1)^d$-element family $\{X_{\bn}: \w{0} \leq  \bn \leq  (m,m,\ldots,m)\}$.

Below a link between the two formulas is described 
{in two ways: a more conceptual one, and one that is shorter but perhaps less intuitive}.
To avoid annoying technicalities, we focus on $d=2$.

\subsubsection{{First approach: counting clusters.}}
{Our aim is to calculate the number of clusters of exceedances over $v_n$ in the window $W:=\{\bk\in\GZ^2 : \mathbf{1} \leq \bk \leq \bN(n)\}$ in two different ways and obtain, as a consequence, the equivalence of (\ref{FIRST}) and (\ref{SECOND}).}

{Let the random set $J_n$ be given as $J_n:=\{\bk\in W : X_{\bk}>v_n\}$
and let $\leftrightarrow$ be the equivalence relation on $J_n$ defined as follows:
\begin{equation*}
\bi \leftrightarrow \bj \quad \textrm{whenever} \quad \exists_{l\in\mathbb{N}}\;\exists_{\bk_0,\bk_1, \ldots, \bk_l,\bk_{l+1}\in J_n, \bk_0=\bi, \bk_{l+1}=\bj} \;\; \max_{0\leq h\leq l}\|\bk_{h+1}-\bk_h\| \leq m,
\end{equation*}
for $\bi,\bj\in J_n$. We define a {\em cluster} as an equivalence class of $\leftrightarrow$ and obtain the partition $\mathcal{C}_n:=J_n/ _{\leftrightarrow}$ of $J_n$ into $\Lambda_n:=\# \mathcal{C}_n$ clusters. We put $\lambda_n := \E (\Lambda_n)$, $\mathcal{C}'_n:=\{C\in\mathcal{C}_n : \max_{\bi, \bj \in C}\|\bi-\bj\| \leq m\}$, $\mathcal{C}''_n:=\mathcal{C}_n \backslash \mathcal{C}'_n$, $\Lambda'_n:=\#\mathcal{C}'_n$ and $\lambda'_n:=\E (\Lambda'_n)$.}

{Let $\Lambda^{(1)}_n$ and $\lambda^{(1)}_n$, associated with the method presented in Section \ref{M_FIRST}, be defined as in Remark \ref{PPA} with $\bp (n):=(m,m)$. Recall that we have $A(m,m)=\{\bj\in\GZ^2 : (-m,-m) \leq \bj \leq (m,m) \textrm{ and } (0,0)\prec \bj\}$. Analogously (see \cite[Remark~2.2]{JS}) we define $\Lambda^{(2)}_n$ and $\lambda^{(2)}_n$ related with the method from Section \ref{M_SECOND} as follows:
\begin{eqnarray*}
\Lambda^{(2)}_n &:=& \sum_{\bk \in W}\sum_{\pmb{\varepsilon}\in\{0,1\}^2} (-1)^{\varepsilon_1 + \varepsilon_2} \mathbbm{1}_{\{M_{\bk+\pmb{\varepsilon}, \bk + (m,m)} > v_n\}},\\
\lambda^{(2)}_n &:=& \E (\Lambda^{(2)}_n) = N^*(n) \sum_{\pmb{\varepsilon}\in\{0,1\}^2} (-1)^{\varepsilon_1 + \varepsilon_2} \P(M_{\pmb{\varepsilon}, (m,m)} > v_n).
\end{eqnarray*}
}

{Assume that $C\in\mathcal{C}'_n$. Let  $B(C,m):=\{\bj\in\mathbb{Z}^2 : \|\bj - \bk\| \leq m \textrm{ for some } \bk \in C\}$ and suppose that $M_{B(C,m) \backslash C} \leq v_n$ holds (which obviously is satisfied in the typical case $B(C,m)\subset W$). Observe that for such $C$ we have
\begin{equation}\label{1_1}
\sum_{\bk \in C} \mathbbm{1}_{\{M_{A(m,m) + \bk}\leq  v_n \}} = \sum_{\bk \in C} \mathbbm{1}_{\{ \bk \textrm{ is the largest element of }C \textrm{ with respect to } \preccurlyeq\}} = 1
\end{equation}
and
\begin{equation}\label{1_2}
\sum_{\bk \in \bar{C}} \sum_{\pmb{\varepsilon}\in\{0,1\}^2} (-1)^{\varepsilon_1 + \varepsilon_2} \mathbbm{1}_{\{M_{\bk+\pmb{\varepsilon},\bk+(m,m)} > v_n \}}
= \mathbbm{1}_{\{M_{\bk(C), \bk(C)+(m,m)} > v_n\}} = 1,
\end{equation}
where $\bar{C}:=\{\bk \in \GZ^2 : \bk + \bi \in C \textrm{ for some } \mathbf{0} \leq \bi \leq (m,m)\}$ and $\bk(C)$ satisfies the~condition $C\subset \{ \bk \in \GZ^2 : \bk(C) \leq \bk \leq \bk(C) + (m,m)\}$.}

{We will show that each of the following equalities holds:
\begin{eqnarray}
\delta_n &:=& \E|\Lambda_n - \Lambda'_n| = \E(\Lambda_n - \Lambda'_n) = o(1); \label{EPS} \\
\delta^{(1)}_n &:=& \E|\Lambda^{(1)}_n - \Lambda'_n| = o(1); \label{EPS_1}\\
\delta^{(2)}_n &:=& \E|\Lambda^{(2)}_n - \Lambda'_n| = o(1). \label{EPS_2}
\end{eqnarray}
This will entail the condition $\lambda_n=\lambda'_n + o(1) = \lambda^{(1)}_n +o(1) = \lambda^{(2)}_n +o(1)$
and complete this section.}

{To show (\ref{EPS}), observe that the event $\{\# \mathcal{C}''_n = l\}$, for $l\in\mathbb{N}_+$, implies that there exist pairs $\bj(i,a), \bj(i,b) \in J_n$, for $i\in\{1,2,\ldots,l\}$, such that $m < \|\bj(i,a) - \bj(i,b)\|\leq 2m$ holds for each $i$ and $\|\bj(i_1,c_1)-\bj(i_2,c_2)\| > m$ is satisfied for $i_1\neq i_2$ and $c_1, c_2\in\{a,b\}$.  Thus we have 
$$ \delta_n  =  \sum_{l=1}^{\infty} l\P(\# \mathcal{C}''_n = l)
 \leq  \sum_{l=1}^{\infty} l \left(N^*(n)\left((4m+1)^2-(2m+1)^2\right) \P(X_\mathbf{0} > v_n)^2\right)^l.$$
Since $q_n:=N^*(n)((4m+1)^2-(2m+1)^2) \P(X_\mathbf{0} > v_n)^2=o(1)$ by (\ref{ASS_STRONGER}), we obtain that
$\delta_n \leq \sum_{l=1}^{\infty} l(q_n)^{l} = q_n (1-q_n)^{-2}$
for all large $n$'s and finally $\delta_n=o(1)$.}

{Before we establish (\ref{EPS_1}) and (\ref{EPS_2}), we will give an upper bound for the probability that a fixed $\bk\in W$ belongs to a large cluster. Note that we have
\begin{eqnarray}\label{OSZ}
\lefteqn{\P\left(\bk \in C \textrm{ for some } C\in\mathcal{C}_n''\right)}\nonumber\\
&=& \P\left(\bk \in C \textrm{ and } \|\bj-\bk\|\leq m \textrm{ for all }\bj\in C, \textrm{ for some } C\in\mathcal{C}_n'' \right) \nonumber\\
&& + \; \P\left(\bk \in C \textrm{ and } \|\bj-\bk\|> m \textrm{ for some }\bj\in C, \textrm{ for some } C\in\mathcal{C}_n''\right) \\
&\leq & \P\left(\|\bi\!-\!\bk\| \leq m \textrm{ and } \|\bj \!-\! \bk\|\leq m \textrm{ and } \|\bi \!-\! \bj\|>m, \textrm{ for some }\bi,\bj\in J_n\right) \nonumber\\
&& + \; \P\left(\bk \in J_n \textrm{ and } m<\|\bk-\bj\|\leq 2m \textrm{ for some }\bj \in J_n\right) \nonumber\\
&\leq & \left((2m+1)^4 + ((4m+1)^2 - (2m+1)^2)\right)\P(X_{\mathbf{0}} > v_n)^2 = a(m)\P(X_{\mathbf{0}} > v_n)^2 \nonumber
\end{eqnarray}
with $a(m):=(2m+1)^4 + 4m(3m+1)$.}

{Applying observation (\ref{1_1}), property (\ref{OSZ}) and taking into account estimation errors for clusters situated near the edges of the window $W$, we obtain
\begin{eqnarray*}
\delta_n^{(1)} & = & \E \left|\sum_{\bk \in  W} \left(\mathbbm{1}_{\{\bk \in \bigcup \mathcal{C}'_n, M_{\bk + A(m,m)}\leq  v_n \}}
+ \mathbbm{1}_{\{\bk \in \bigcup \mathcal{C}''_n , M_{\bk + A(m,m)}\leq  v_n\}}\right) - \Lambda'_n \right|\\
& \leq & \E\left|\sum_{C\in\mathcal{C}'_n} \sum_{\bk\in C} \mathbbm{1}_{\{M_{\bk + A(m,m)}\leq  v_n\}} - \Lambda'_n\right|
+  \E \left(\sum_{\bk \in W}  \mathbbm{1}_{\{\bk \in \bigcup \mathcal{C}''_n \}}\right)\\
& \leq &  2m(N_1(n)+N_2(n)) \P(X_{\mathbf{0}} > v_n)+ a(m)N^*(n)\P(X_{\mathbf{0}} > v_n)^2,
\end{eqnarray*}
which combined with assumption (\ref{ASS_STRONGER}) implies (\ref{EPS_1}).
Quite similarly, using (\ref{1_2}) instead of (\ref{1_1}) and writing $\bar{W}:=\{\bk\in \GZ^2 : \mathbf{1}-(m,m) \leq \bk \leq \bN(n)\}$, we conclude that
\begin{eqnarray*}
\delta_n^{(2)} 
& \leq & \E \left| \sum_{\bk\in \bar{W} } \mathbbm{1}_{\{\bk\in\bar{C} \textrm{ for some }C\in\mathcal{C}'_n \}} \sum_{\pmb{\varepsilon}\in\{0,1\}^2} (-1)^{\varepsilon_1 + \varepsilon_2} \mathbbm{1}_{\{M_{\bk+\pmb{\varepsilon}, \bk + (m,m)} > v_n\}} - \Lambda'_n\right| \\
&& \hspace{-0.5cm} + \;\, \E \left(\sum_{\bk\in\bar{W}\backslash W} \sum_{\pmb{\varepsilon}\in\{0,1\}^2} \left| (-1)^{\varepsilon_1 \! + \varepsilon_2} \mathbbm{1}_{\{M_{\bk+\pmb{\varepsilon}, \bk + (m,m)} > v_n\}}\right|\right) \\
&& \hspace{-0.5cm} + \;\, 2\E \left(\sum_{\bk\in W} \mathbbm{1}_{\{\bk\in\bar{C} \textrm{ for some }C\in\mathcal{C}''_n \}} \right)\\
& \leq & \E\left| \sum_{C\in \mathcal{C}'_n} \sum_{\bk\in \bar{C}} \sum_{\pmb{\varepsilon}\in\{0,1\}^2} (-1)^{\varepsilon_1 + \varepsilon_2} \mathbbm{1}_{\{M_{\bk+\pmb{\varepsilon}, \bk + (m,m)} > v_n\}} - \Lambda'_n  \right|\\
&& \hspace{-0.5cm} +  \sum_{\bk\in\bar{W}\backslash W} (2m+1)\P(X_{\mathbf{0}} > v_n) + 2N^*(n)\P\left(\bk\in\bar{C} \textrm{ for some }C\in\mathcal{C}''_n\right)\\
& \leq & 2m(2m+1)(N_1(n)+N_2(n))\P(X_\mathbf{0}>v_n)\\
&& \hspace{-0.5cm} + \;\,  m(2m \! + \!\! 1)(N_1(n) \! + \! N_2(n) \! + \! m)\P(X_\mathbf{0} \! > \! v_n) \! + \! 2(m \! + \!\! 1)^2a(m)N^*(n)\P(X_\mathbf{0} \! > \! v_n)^2.
\end{eqnarray*}
Since the right-hand side tends to zero by (\ref{ASS_STRONGER}), property (\ref{EPS_2}) follows.}

\subsubsection{Second approach: direct verification.} In this part we assume that $\preccurlyeq$ is the lexicographic order on $\GZ^2$. Let us notice that
\begin{eqnarray*}
\lefteqn{\P\!\left(M_{{(0,0)},(m,m)} \!\! > \! v_n\right) \! - \! \P\!\left(M_{(1,0),(m,m)} \!\! > \! v_n\right) \! - \! \P\!\left(M_{(0,1),(m,m)} \!\! > \! v_n\right) \! + \! \P\!\left(M_{(1,1),(m,m)} \!\! > \! v_n\right)}\\
& = & \P\left(X_{(0,0)} > v_n, M_{R((m,m))} \leq  v_n\right)\\
&& \hspace{3cm} - \;\; \P\left(M_{(0,1), (0,m){}} > v_n, M_{(1,0),(m,0)} >v_n, M_{(1,1), (m,m)} \leq  v_n\right)
\end{eqnarray*}
is true with $R((p_1,p_2)):=A((p_1,p_2))\cap \GN^2$, where on the left-hand side of the equation the sum of probabilities from (\ref{SECOND}) for $d=2$ appears.
Next, let us look at the exponent in (\ref{FIRST}) and observe that
\begin{eqnarray*}\
\lefteqn{\P\left(X_{(0,0)} > v_n,M_{A((m,m))} \leq  v_n\right)}\\
& = & \P\left(X_{(0,0)} > v_n, M_{R((m,m))} \leq  v_n\right)\\
&& \hspace{3cm} - \;\;  \P\left(X_{(0,0)} > v_n, M_{R((m,m))} \leq  v_n, M_{(1,-m),(m,-1)} > v_n\right)
\end{eqnarray*}
holds and, moreover, the second summand of the right-hand side satisfies
\begin{eqnarray*}
\lefteqn{\P\left(X_{(0,0)} > v_n, M_{R((m,m))} \leq  v_n, M_{(1,-m),(m,-1)} > v_n\right)} \\
& = & \!\! \sum_{l=1}^{m} \P\left(X_{(0,0)} > v_n, M_{R((m,m))} \leq  v_n, M_{(1,-l),(m,-l)} > v_n, M_{(1,-l+1), (m,-1)} \leq  v_n\right) \\
& = & \!\! \sum_{l=1}^{m} \P\left(X_{(0,0)} > v_n, M_{R((m,m-l))} \leq  v_n, M_{(1,-l),(m,-l)} > v_n, M_{(1,-l+1), (m,-1)} \leq  v_n\right) \\
&& \hspace{10.6cm} + \;\; o\left(n^{-2}\right)\\
& = & \!\! \sum_{l=1}^{m} \P\left(X_{(0,l)} > v_n, M_{(0,l) + R((m,m-l))} \leq  v_n, M_{(1,0),(m,0)} > v_n, M_{(1,1), (m,l-1)} \leq  v_n\right) \\
&& \hspace{10.6cm} + \;\; o\left(n^{-2}\right) \\
& = &  \!\! \P\left(M_{(0,1), (0,m)} > v_n, M_{(1,0),(m,0)} >v_n, M_{(1,1), (m,m)} \leq  v_n \right) +o\left(n^{-2}\right).
\end{eqnarray*}
In the above statement probabilities of mutually exclusive events are summed up and \mbox{$m$-dependence}, assumption (\ref{ASS_STRONGER}) and stationarity are applied.
Finally, we obtain
\begin{eqnarray*}
\lefteqn{\P\!\left(M_{{(0,0)},(m,m)} \!\! > \! v_n\right) \! - \! \P\!\left(M_{(1,0),(m,m)} \!\! > \! v_n\right) \! - \! \P\!\left(M_{(0,1),(m,m)} \!\! > \! v_n\right) \! + \! \P\!\left(M_{(1,1),(m,m)} \!\! > \! v_n\right)}\\
&& \hspace{5.3cm} = \;\; \P\left(X_{(0,0)} > v_n,M_{A((m,m))} \leq  v_n\right) + o\left(n^{-2}\right).
\end{eqnarray*}
Summarizing, we have confirmed that both presented methods lead to the same result.

\begin{rema}
The above reasoning for $m$-dependent fields can also be applied in the general setting. 
Suppose that formula (\ref{OUR_FORMULA}), with $\preccurlyeq$ the lexicographic order on $\GZ^2$, describes the asymptotics of partial maxima of the stationary field $\{X_\bn : \bn\in\GZ^2\}$.
Then
\begin{equation}\label{SECOND_GEN}
\P \left(M_{\bN (n)}  \leq   v_n\right)=\exp\left( -n^2  \sum_{\pmb{\varepsilon}\in\{0,1\}^2}  (-1)^{\varepsilon_1+\varepsilon_2}\P \left( M_{\pmb{\varepsilon},\bp(n)}  >  v_n\right) \right)  +  o(1)
\end{equation}
holds if and only if $\{X_\bn\}$ satisfies the following condition:
\begin{equation}\label{COND_PROP}
\sum_{l=1}^{p_2(n)} \P\left(X_{\mathbf{0}} > v_n, M_{U_l(\bp (n))} >  v_n, M_{V_l(\bp (n))} > v_n, M_{W_l(\bp (n))} \leq  v_n\right) = o(n^{-2}),
\end{equation}
where $U_l(\bp):=\{0,\ldots, p_1\} \times \{p_2-l+1, \ldots, p_2\}$, $V_l(\bp):=\{1,\ldots, p_1\} \times \{-l\}$ and $W_l(\bp):= A(\bp) \cap (\GZ\times \{-l+1,\ldots, p_2-l\})$.
Formula (\ref{SECOND_GEN}) generalizes (\ref{SECOND}). In the present section we have used the fact that $m$-dependent fields satisfy (\ref{COND_PROP}) with $\bp(n):=(m,m)$.
\end{rema}

\section{Example: moving maxima}\label{MM_SEC}
Below, we use the results from Sections \ref{MAIN} and \ref{SEC_EI} to describe the asymptotics of partial maxima for the moving maximum field. We note that approaches to the problem using different methods can be found in Basrak and Tafro \cite{B-T2014} or Jakubowski and Soja-Kukie\l a~\cite{JS}. In the first paper compound Poisson point process approximation is applied while in the second one the authors combine a Bonferroni-like inequality and max-$m$-approximability.

In the following $\{Z_\bn\}$ is an array of i.i.d. random variables satisfying 
\begin{equation}\label{reg}
\P(|Z_\mathbf{0}| > x) = x^{-\alpha} L(x),
\end{equation}
for some index $\alpha >0$ and slowly varying function $L$, and
\begin{equation}\label{balance}
\frac{\P(Z_\mathbf{0} > x)}{\P(|Z_\mathbf{0}| > x)}=p\quad\text{as}\quad x\to\infty, \quad\text{for some}\quad  p\in[0,1].
\end{equation}
We define $a_n:=\inf\{y>0:\P(|Z_\mathbf{0}| > y) \leq n^{-d}\}$ and $v_n:=a_n v$ with fixed $v>0$. Then
$$n^d \P(|Z_\mathbf{0}| > v_n) \to v^{-\alpha}\quad\text{as}\quad n\to\infty.$$
Let us consider the {\em moving maximum field} $\{X_\bn\}$ defined as
$$X_\bn=\sup_{\bj\in\GZ ^d}c_\bj Z_{\bn+\bj},$$
where $c_\bj\in\GR$, not all equal to zero, satisfy
\begin{equation}\label{mm_ass}
\sum_{\bj\in\GZ ^d} |c_\bj|^\beta <\infty \quad \text{for some} \quad 0<\beta < \alpha.
\end{equation}
From Cline \cite[Lemma 2.2]{CLINE} it follows that the field $\{X_\bn\}$ is well defined and
\begin{eqnarray}\label{tail_mm}
\lim_{x\to\infty}\frac{\P(X_\mathbf{0}>x)}{\P(|Z_\mathbf{0}|>x)} 
& = &\lim_{x\to\infty}\frac{\P\left(\sup_{\bj\in\GZ ^d}c_\bj  Z_{\bj}>x\right)}{\P(|Z_\mathbf{0}|>x)} \nonumber \\
& =& \lim_{x\to\infty}\frac{\sum_{\bj\in\GZ^d} \P(c_\bj  Z_{\bj}>x)}{\P(|Z_\mathbf{0}|>x)} 
= p  \sum_{c_\bj  > 0}  c_\bj ^\alpha
+ q  \sum_{c_\bj  < 0} |c_\bj |^\alpha.
\end{eqnarray}
with $q:=1-p$.

{Since the moving maximum field is \emph{max-$m$-approximable}, there exists a \emph{phantom distribution function} for $\{X_\bn\}$ (see Jakubowski and Soja-Kukie\l a \cite{JS}) and hence the field is weakly dependent in the sense of~(\ref{WEAK_DEP}).}
We will apply Theorem \ref{THEO_MAIN} with $\preccurlyeq$ being the lexicographic order on $\GZ^d$ to describe the asymptotics of partial maxima. Let us observe that the exponent in~(\ref{CONV_MAIN}) satisfies
\begin{eqnarray*}
\lefteqn{n^d \P\left(X_{\w{0}} > v_n, M_{A(\bp(n))}\leq  v_n\right)}\\ 
& = & \!\!\! n^d\P\left(\bigcup_{\bj\in\GZ^d} \{c_\bj Z_\bj > v_n \}, \bigcap_{\bk\in\GZ^d} \left\{\max_{\bi\in A(\bp(n))} (c_{\bk - \bi} Z_\bk) \leq v_n\right\} \right)\\
& = & \!\!\! n^d \sum_{\bj \in \GZ^d} \P\left( c_\bj Z_\bj > v_n , \bigcap_{\bk\in\GZ^d} \left\{\max_{\bi\in A(\bp(n))} (c_{\bk - \bi} Z_\bk) \leq v_n\right\}\right) + o(1)\\
& = & \!\!\! n^d \sum_{\bj \in \GZ^d} \P\left( c_\bj Z_\bj > v_n \geq \max_{\bi\in A(\bp(n))} (c_{\bj - \bi} Z_\bj), \bigcap_{\bk \neq \bj} \left\{\max_{\bi\in A(\bp(n))} (c_{\bk - \bi} Z_\bk) \leq v_n\right\}\right) +  o(1)\\
& = & \!\!\! n^d \sum_{\bj \in \GZ^d} \P \left( c_\bj Z_\bj \! > \! v_n \! \geq \! \max_{\bi\in A(\bp(n))} \! (c_{\bj - \bi} Z_\bj)\right) \P \left(\bigcap_{\bk \neq \bj} \left\{\max_{\bi\in A(\bp(n))} \! (c_{\bk - \bi} Z_\bk) \! \leq \! v_n\right\}\right) \! + \! o(1),
\end{eqnarray*}
as $n\to\infty$, where the second equality follows from (\ref{tail_mm}) combined with the choice of $\{v_n\}$ and the last one is a consequence of the independence of $Z_\bj$ for $\bj\in\GZ^d$. Note that we have
\begin{multline*}
\P\left(\bigcap_{\bk \neq \bj} \left\{\max_{\bi\in A(\bp(n))} (c_{\bk - \bi} Z_\bk) \leq v_n\right\}\right)
 \geq  \P\left(\bigcap_{\bk \in \GZ^d} \left\{\max_{\bi\in A(\bp(n))} (c_{\bk - \bi} Z_\bk) \leq v_n\right\}\right) \\
 \geq  \P \left(M_{A(\bp(n))}\leq  v_n\right) \geq 1- o(n^d)\P(X_{\mathbf{0}} >v_n) = 1 +  o(1).
\end{multline*}
Moreover, for $p_{\min}(n):=\min \{p_l(n) : 1\leq l\leq d\}$ and for $q(n) \in \GN$ chosen so that $q(n)\to\infty$, $q(n) \leq p_{\min}(n)/2$ and $q(n)^d n^d \P(\max\{ c_{\bi}Z_{\mathbf{0}} : \|\bi\| >p_{\min}(n)/2 \} > v_n) \to 0$, it follows that
\begin{eqnarray*}
\lefteqn{\hspace{-0.75cm} \left| n^d \sum_{\bj \in \GZ^d} \P\left( c_\bj Z_\bj > v_n \geq \max_{\bi\in A(\bp(n))} (c_{\bj - \bi} Z_\bj)\right)
-  n^d \sum_{\bj\in\GZ^d} \P\left( c_\bj Z_\bj > v_n \geq \sup_{\mathbf{0} \prec \bi} (c_{\bj - \bi} Z_\bj)\right) \right| }\\
& \leq & n^d \sum_{\bj \in \GZ^d} \P\left( c_\bj Z_\bj > v_n, \sup_{\mathbf{0} \prec \bi \notin A(\bp(n))} (c_{\bj - \bi} Z_\bj) > v_n \right)\\
& \leq & n^d \sum_{\bj \in \GZ^d} \P\left( c_\bj Z_\bj > v_n, \sup_{\|\bi \| > p_{\min}(n)} (c_{\bj - \bi} Z_\bj) > v_n \right)\\
& \leq & n^d\sum_{\|\bj\| \leq q(n)} \P\left(\sup_{\|\bi \| > p_{\min}(n)} (c_{\bj - \bi} Z_\bj) > v_n \right) + n^d \sum_{\|\bj\| > q(n)} \P\left( c_\bj Z_\bj > v_n\right)\\
& \leq & {n^d}(2 q(n)+1)^d \P\left(\sup_{\|\bi \| > p_{\min}(n)/2} (c_{\bi} Z_{\mathbf{0}}) > v_n \right) + n^d \sum_{\|\bj\| > q(n)} \P\left( c_\bj Z_\bj > v_n\right)
\end{eqnarray*}
The first summand on the right-hand side tends to zero due to the choice of $q(n)$ and the second one tends to zero by properties (\ref{mm_ass}), (\ref{tail_mm}) and the definition of $v_n$. We conclude that
\begin{multline*}
n^d \P\left(X_{\w{0}} > v_n, M_{A(\bp(n))}\leq  v_n\right) \\
= \left(n^d \sum_{\bj \in \GZ^d} \P\left( c_\bj Z_\bj > v_n \geq \sup_{\mathbf{0} \prec \bi} (c_{\bj - \bi} Z_\bj)\right) + o(1)\right) (1+ o(1)) +  o(1).
\end{multline*}
To complete the above calculation, it is sufficient to observe that
\begin{multline*}
n^d \sum_{\bj \in \GZ^d} \P\left( c_\bj Z_\bj > v_n \geq \sup_{\mathbf{0} \prec \bi} (c_{\bj - \bi} Z_\bj)\right)
=  n^d \sum_{\bj \in \GZ^d} \P\left( c_\bj Z_{\mathbf{0}} > v_n \geq \sup_{\bi \prec \bj} (c_{\bi} Z_{\mathbf{0}})\right)\\
=  n^d \P\left(\sup_{\bj \in \GZ^d} (c_\bj Z_{\mathbf{0}}) > v_n\right) \to (p(c^+)^{\alpha} + q(c^-)^{\alpha})v^{-\alpha},
\end{multline*}
with $c^+:=\max_{\bi \in\GZ ^d}\max\{c_\bi,0\}$ and $c^-:=\max_{\bi \in\GZ ^d} \max\{-c_\bi,0\}$.
By (\ref{CONV_MAIN}) we obtain
$$ \P\left(M_{\bN(n)} \leq  v_n\right) \to \exp(-(p(c^+)^{\alpha} + q(c^-)^{\alpha})v^{-\alpha}), \quad \text{as} \quad n\to\infty. $$
Applying formula (\ref{F1}) and property (\ref{tail_mm}), we calculate the extremal index of $\{X_\bn\}$ as follows
\begin{equation*}
\theta=\frac{p(c^+)^\alpha + q(c^-)^\alpha}{ p\sum_{c_\bj > 0}c_\bj ^\alpha
+ q \sum_{c_\bj  < 0}|c_\bj |^\alpha},
\end{equation*}
whenever the denominator is positive, which is the only interesting case.

\acks
The author would like to thank the referees and the associate editor for comments and suggestions which improved the manuscript significantly.

\end{document}